\documentclass{amsart}

\usepackage[dvipdfmx]{graphicx}

\newtheorem{thm}{Theorem}[section]
\newtheorem{cor}[thm]{Corollary}

\newtheorem{lem}[thm]{Lemma}

\theoremstyle{remark}

\theoremstyle{definition}

\numberwithin{equation}{section}
\numberwithin{thm}{section}

\theoremstyle{plain}
\newtheorem{clm}{Claim}

\begin{document}


\title[Decomposing Heegaard splittings]{Decomposing Heegaard splittings along separating incompressible surfaces in  3-manifolds}

\author{Kazuhiro Ichihara}
\address{Department of Mathematics, College of Humanities and Sciences, Nihon University, 3-25-40 Sakurajosui, Setagaya-ku, Tokyo 156-8550, Japan.}
\email{ichihara@math.chs.nihon-u.ac.jp}

\author{Makoto Ozawa}
\address{Department of Natural Sciences, Faculty of Arts and Sciences, Komazawa University, 1-23-1 Komazawa, Setagaya-ku, Tokyo, 154-8525, Japan.}
\email{w3c@komazawa-u.ac.jp}

\author{J. Hyam Rubinstein}
\address{Department of Mathematics and Statistics, The University of Melbourne, VIC 3010, Australia.}
\email{joachim@unimelb.edu.au}

\thanks{The first author and the second author are partially supported by Grant-in-Aids for Scientific Research (C) (No. 26400100, 17K05262), The Ministry of Education, Culture, Sports, Science and Technology, Japan, respectively. The third author is partially supported under the Australian Research Council's Discovery funding scheme (project number DP130103694).}

\dedicatory{Dedicated to Professor Tsuyoshi Kobayashi on the occasion of his 60th birthday}

\begin{abstract}
In this paper, by putting a separating incompressible surface in a 3-manifold 
into Morse position relative to the height function associated to a strongly irreducible Heegaard splitting, 
we show that an incompressible subsurface of the Heegaard splitting can be found, by decomposing the 3-manifold along the separating surface. 
Further if the Heegaard surface is of Hempel distance at least 4, then there is a pair of such subsurfaces on both sides of the given separating surface. This gives a particularly simple hierarchy for the 3-manifold.
\end{abstract}

\maketitle

\section{Introduction}

In this paper, we consider separating incompressible surfaces embedded in a closed orientable irreducible 3-manifold with Heegaard splittings. 
Here a \textit{Heegaard splitting} of a closed orientable 3-manifold $M$ is 
a splitting along a closed orientable surface $S$ embedded in $M$, 
called a \textit{Heegaard surface}, into two handlebodies. 

Note that a Heegaard splitting of a 3-manifold $M$ with a Heegaard surface $S$ induces a height function $h: M \to [0,1]$ on $M$. In particular there is a singular foliation $\{h^{-1}(t)= S_t \}_{0<t<1}$ of $M$ with each $S_t$ homeomorphic to $S$ and $h^{-1}(0), h^{-1}(1)$ graphs which are spines of the two handlebodies bounded by $S= h^{-1}(1/2)$.
Relative to $h$, given an incompressible surface $J$ in $M$, it is well-known that one can put $J$ into Morse position, namely there are finitely many singularities of $J$ relative to the foliation $\{S_t\}$ and these are all of simple saddle type. 
By using this, if the Heegaard splitting is strongly irreducible, we show that there always exist incompressible level subsurfaces as follows. 

\begin{thm}\label{MainThm}
Let $M$ be a closed irreducible orientable 3-manifold admitting a strongly irreducible Heegaard splitting. 
Let $\{ S_t \}_{0<t<1}$ be a singular foliation of $M$ associated to the height function $h$ for the splitting. Furthermore let $J$ be a separating closed orientable incompressible surface in $M$ which cuts $M$ open into $M_+$ and $M_-$. Assume that $J$ is in Morse position relative to $h$,
Then there exists a non-critical value of $t$ to that the level surface $S_t$ satisfies one of $S_t \cap M_+$ or $S_t \cap M_-$ is incompressible in $M_+$ or $M_-$ respectively. 
Furthermore, if the Heegaard splitting is of Hempel distance at least $4$, then there is a non-critical value of $t$ so that both 
$S_t \cap M_+$ and $S_t \cap M_-$ are incompressible in each of $M_+$ and $M_-$. 
\end{thm}

The first assertion of the theorem gives an alternative proof of \cite[Proposition 2.6]{KobayashiQiu}, and the second assertion also gives 
that of a recent result in \cite{Saito}. 

Here a Heegaard splitting of $M$ is called \textit{strongly irreducible} if for its Heegaard surface $S$, every compressing disk on one side of $S$ meets every compressing disk on the other side of $S$. 
Also the \textit{Hempel distance} of a Heegaard splitting for $M$ is defined as follows: 
For its Heegaard surface $S$, consider the collections of curves $\mathcal C$ and $\mathcal C^{\prime}$ which bound compressing disks for $S$ in the two regions, which are handlebodies on either side of $S$. 
A path between these collections is a sequence of essential simple closed curves $C=C_0,C_1, \dots C_k$ so that each pair $C_i, C_{i+1}$ are disjoint and $C_0  \in \mathcal C, C_k \in \mathcal C^{\prime}$. 
The Hempel distance is then the smallest value of $k$ amongst all such sequences. 
See \cite{Hempel} for the original definition.

\section{Proofs}

We first prepare the following lemma essentially given in \cite[Lemma 3.2]{Ozawa}. See also \cite{Hartshorn}. 

\begin{lem}\label{lem}
Let $M$ be a closed irreducible orientable 3-manifold admitting a strongly irreducible Heegaard splitting $M = V \cup_S W$. 
Form a height function $h : M \to [0,1]$ associated to the splitting with $S= h^{-1} (1/2)$,  $S_t$ homeomorphic to $S$ for $0 < t < 1$, and $h^{-1} (0)$, $h^{-1} (1)$ are the spine of the handlebodies $V$, $W$ respectively 
Then a closed incompressible surface $J$ embedded in $M$ can be isotoped so that the following conditions are satisfied. 
\begin{enumerate}
\item $J$ intersects both spines of $V$, $W$ transversely.
\item $J$ has only simple saddle points with respect to the height function $h$ for $0< t < 1$ at mutually distinct levels. 
\item  $J$ intersects each level Heegaard surface in essential curves. 
\end{enumerate}
\end{lem}

\begin{proof}
First we assume that $J$ and both spines $h^{-1} (0)$, $h^{-1} (1)$ are in general position.
Then (1) is satisfied.

Next we assume that $h|_J$ is a Morse function for $0<t<1$, that is, it has only finitely many critical points, all non-degenerate, and with all critical values distinct.

We assume that the sum of $|J\cap (h^{-1} (0)\cup h^{-1} (1))|$ and the number of critical points for $0<t<1$ is minimal up to isotopy of $J$.
Now suppose without loss of generality that there exists a maximal point of $J$ for $0<t<1$.
By the same argument in \cite[Lemma 3.2]{Ozawa}, we have a contradiction on the minimality.

The key steps of the proof are follows. First we look at the lowest maximal point $a$ of $J$, and consider the ``maximal horizontally $\partial$-parallel subsurface'' $J_a$ of $J$ containing $a$, which was defined in \cite{Ozawa} (See Figure 12 in \cite{Ozawa}). Next we consider the band at the saddle point $p$ which is contained in $\partial J_a$. By the maximality of $J_a$, the minimality, the incompressibility of $J$ and the irreducibility of $M$, we have a contradiction for any cases of the band.
Hence (2) is satisfied.

Finally suppose that there exists a loop $l$ of $J\cap h^{-1}(t)$ for $0<t<1$ which bounds a disk in $h^{-1}(t)$.
By the incompressibility of $J$ and the irreducibility of $M$, it follows that there exists a maximal or minimal point of $J$ for $0<t<1$.
This contradicts to (2).
Hence (3) is satisfied.
\end{proof}

We remark that as $t$ increases, when $S_t$ passes a saddle point of $J$, a band of $S_t$ is pushed across $J$ from one side of $J$ into the other, 
in other words, a band sum occurs for some curves in $S_t \cap J$. 
See Appendix to \cite{Haken} for a very elegant discussion of this procedure. 
We will use the terminology that $J$ is in Morse position relative to the height function $h$ to mean that the conditions of Lemma~\ref{lem} are satisfied. 

Now the next theorem gives a proof of the first assertion of Theorem~\ref{MainThm}. 

\begin{thm}
Consider a separating closed orientable incompressible surface $J$ in a closed irreducible orientable 3-manifold $M$ admitting a strongly irreducible Heegaard splitting. 
Denote the two sides of $J$ as $M_+, M_-$, and consider a singular foliation $\{ S_t \}_{0<t<1}$ of $M$ associated to a height function for the Heegaard splitting so that $J$ is in Morse position relative to $S_t$. 
Then either;
\begin{itemize}

\item  there is some non critical level $S_t$ so that $S_t \cap M_+$  is incompressible and $S_{t} \cap M_-$ has compressing disks on both sides of $S_{t}$, or the same with $M_+, M_-$ interchanged.

\item there is a critical level $\hat t$ so that $S_{t} \cap M_+$ is incompressible for $t<{\hat t}$ and $t$ close to ${\hat t}$, and $S_{t'} \cap M_-$  is incompressible for $t' > {\hat t}$ and $t'$ close to ${\hat t}$, or the same with $M_+$, $M_-$ interchanged. Moreover $S_t \cap M_-$ has a compressible disk below $S_t$ for $t < {\hat t}$ and $t$ close to ${\hat t}$ and $S_t \cap M_+$ has a compressible disk above $S_t$ for $t > {\hat t}$ and $t$ close to ${\hat t}$.

\item there is a critical level ${\hat t}$ so that both $S_{t} \cap M_+$ and $S_{t} \cap M_-$ are incompressible for $t<{\hat t}$ and $t$ arbitrarily close to ${\hat t}$.
\end{itemize}

\end{thm}

\begin{proof}
By Lemma~\ref{lem}, we assume that $J$ satisfies the conditions described in the lemma. 

First suppose that at some level $t$, one of $S_t \cap M_+$ or $S_t \cap M_-$ has compressing disks on both sides. 
Since the Heegaard splitting is strongly irreducible, either $S_t \cap M_-$ or $S_t \cap M_+$, respectively must be incompressible. 
So the first case of the theorem holds. 

On the other hand, we assume that neither $S_t \cap M_+$ nor $S_t \cap M_-$ has compressing disks on both sides, for any value of $t$. 
We know that for $t$ small, there are compressing disks for $S_t \cap M_+$ and $S_t \cap M_-$ in $H^0_t$, whereas for $t$ close to $1$, there are compressing disks for $S_t \cap M_+$ and $S_t \cap M_-$ in $H^1_t$. 
Here we denote the two handlebodies obtained by splitting $M$ open along $S_t$ by $H^0_t$ (below $S_t$) and $H^1_t$ (above $S_t$). 
Thus, there exists some level $u$ ($0<u<1$) such that for $t<u$, any compressing disk for $S_t \cap M_+$ and $S_t \cap M_-$ lies in $H^0_t$, whereas for $t>u$, any compressing disk for $S_t \cap M_+$ and $S_t \cap M_-$ lies in $H^1_t$. 
Then, since $J$ has only saddle critical points with respect to the height function, there exists a critical level ${\hat t} \ge u$ at which a band sum occurs which produces the first compressing disk for $S_{t'} \cap M_+$ or $S_{t'} \cap M_-$ in $H^1_{t'}$ for a regular value $t' > \hat t$. 

Without loss of generality, we may assume that at the level $\hat t$, the side on which the band leaves is the $M_+$ side, and the side where the band is received is the $M_-$ side. 

\begin{clm}\label{clm1}
For $t$ very close to $\hat t$ with $t < \hat t$ and $t'$ very close to $\hat t$ with $t' > \hat t$, 
if a compressing disk exists for $S_{t'} \cap M_+$ in $H^1_{t'}$ (resp. for $S_t \cap M_-$ in $H^0_t$), 
then there is a compressing disk in $S_t \cap M_+$ in $H^1_t$ (resp. for $S_{t'} \cap M_-$ in $H^1_{t'}$).
\end{clm}
\begin{proof}
By performing a band sum, $S_t \cap M_+$ is thinned to produce $S_{t'} \cap M_+$, whereas $S_t \cap M_-$ is thickened to form $S_{t'} \cap M_+$. 
See Figure~\ref{fig1}. 
\end{proof}

\begin{figure}[htbp]
	\begin{center}
	\begin{tabular}{cc}
	\includegraphics[trim=0mm 0mm 0mm 0mm, width=.53\linewidth]{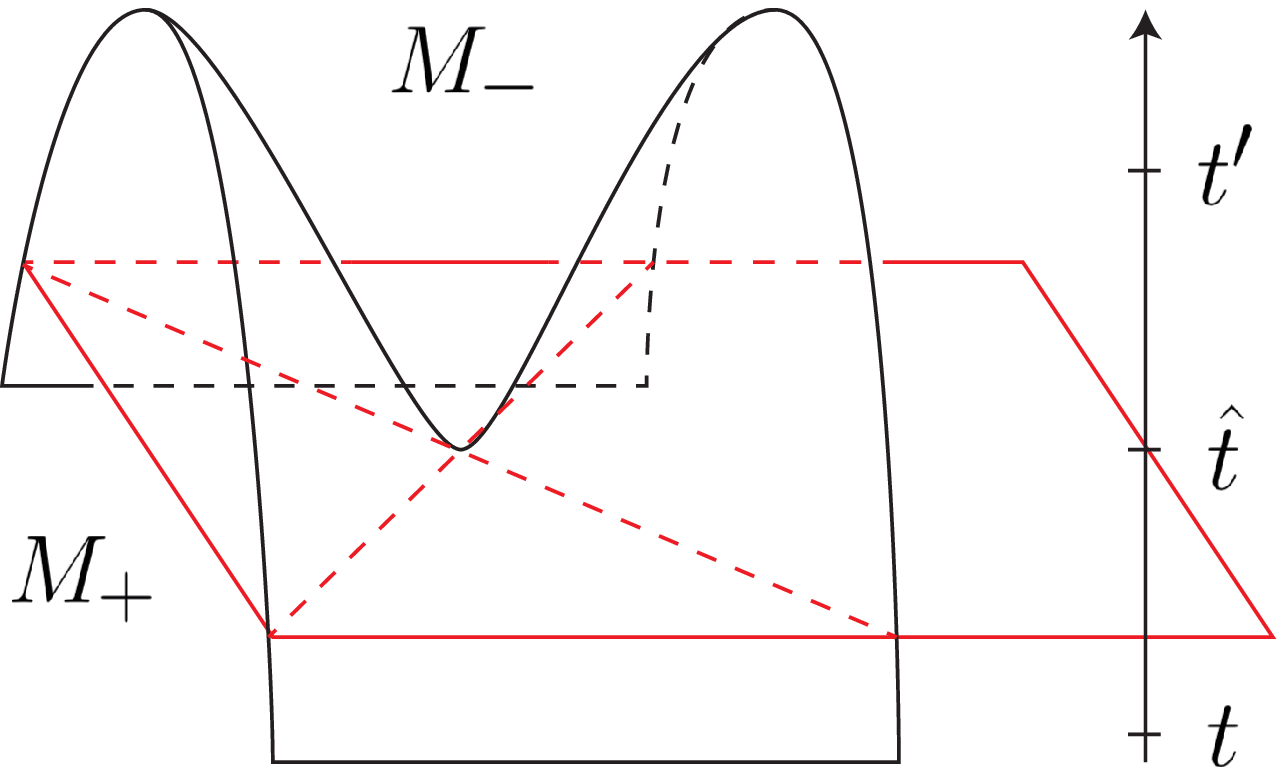}&
	\includegraphics[trim=0mm 0mm 0mm 0mm, width=.37\linewidth]{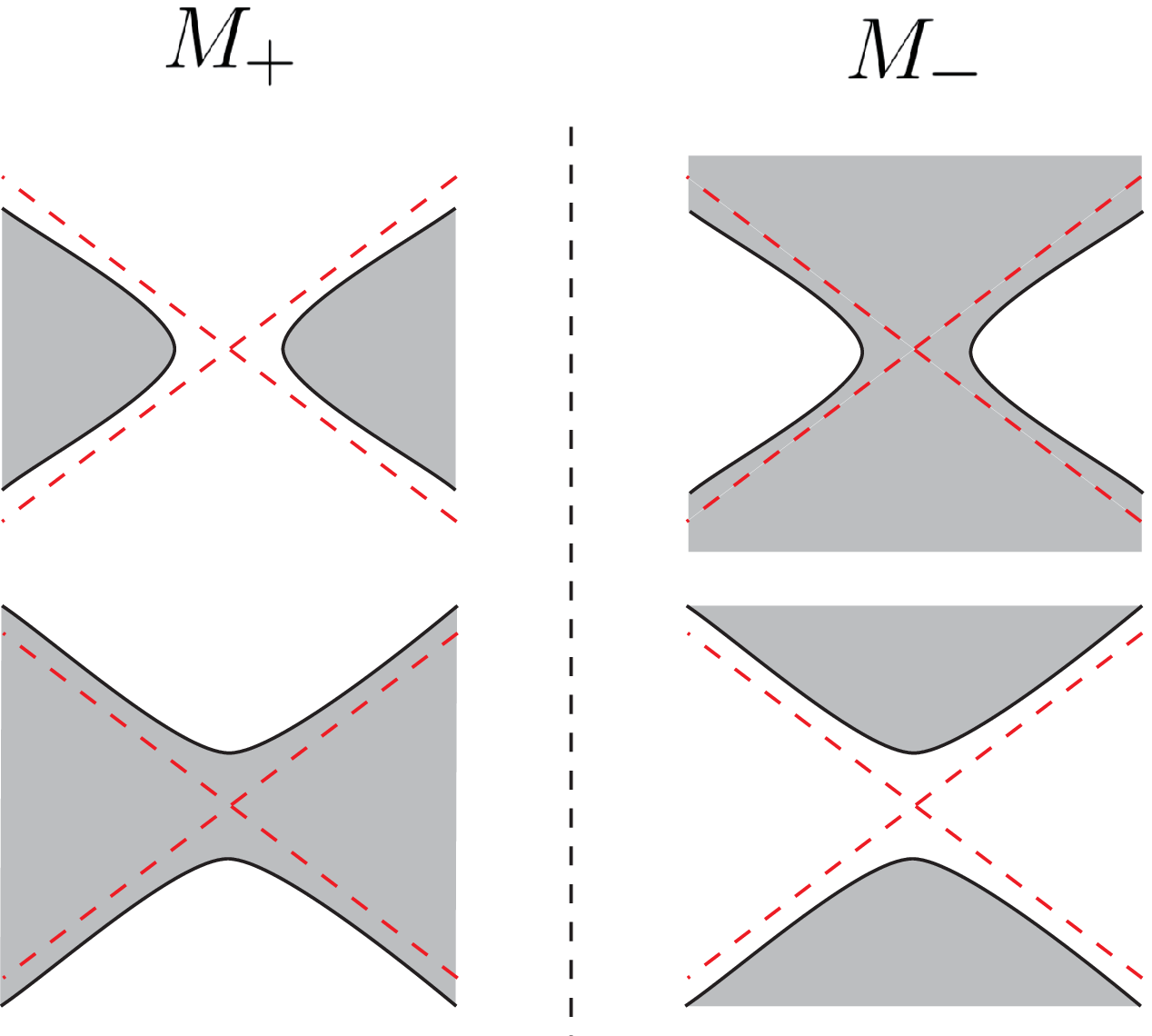}\\
	\end{tabular}
	\end{center}
	\caption{}\label{fig1}
	\label{saddle}
\end{figure}

Now we have the two possibilities; for $t$ very close to $\hat t$ and $t < \hat t$, either $S_t \cap M_+$ has a compressing disk in $H^0_t$, or not. 

In the first case, we see the following. 

\begin{clm}\label{clm2}
There are no compressing disks for $S_{t'} \cap M_+$ in $H^1_{t'}$. 
\end{clm}
\begin{proof}
If such a compressing disk exists in $H^1_{t'}$, then there exists a compressing disk in $S_t \cap M_+$ in $H^1_t$ for $t < {\hat t}$ by Claim~\ref{clm1}, contradicting the assumption that the first compressing disk for $S_{t'} \cap M_+$ in $H^1_{t'}$ appears for $t' > \hat t$. 
\end{proof}

Thus the side $M_-$ must be where the first compressing disk in $H^1_{t'}$ appears at the level $t'$ ($t < \hat{t} < t'$). 
That is, $S_{t'} \cap M_-$ has a compressing disk in $H^1_{t'}$. 
Then, there are no compressing disks for $S_{t'} \cap M_+$ in $H^0_{t'}$, for $S$ is strongly irreducible. 
Together with Claim~\ref{clm2}, we conclude that, for $t' > \hat t$ and $t'$ very close to $\hat t$, there are no compressing disks for $S_{t'} \cap M_+$ in $M_+$, i.e., $S_{t'} \cap M_+$ is incompressible in $M_+$.

Also, at the level $t$ ($t < \hat{t} < t'$), there cannot be any compressing disks in $S_t \cap M_-$ in $H^0_t$. Because if such a compressible disk exists, then it gives a compressing disk in $S_{t'} \cap M_-$ in $H^0_{t'}$ by Claim~\ref{clm1}, contradicting that $S_{t'} \cap M_-$ does not have compressing disks on both sides. 
Also, at the level $t$, there cannot be any compressing disks in $S_t \cap M_-$ in $H^1_t$, for the first compressing disk for $S_{t'} \cap M_-$ in $H^1_{t'}$ appears for $t' > \hat t$. 
We conclude that there cannot be any compressing disks for $S_t \cap M_-$ in $M_-$ at the level $t$, i.e., $S_t \cap M_-$ is incompressible in $M_-$. 
This gives the second case of the theorem. 

Finally the third case occurs when $S_t \cap M_+$ has no compressing disk in $H^0_t$ for $t$ very close to $\hat t$ and $t < \hat t$. 
In this case, there are no compressing disk for $S_t \cap M_+$ in $H^1_t$, for the first compressing disk for $S_{t'} \cap M_+$ in $H^1_{t'}$ appears for $t' > \hat t$. 
This implies that $S_t \cap M_+$ is incompressible in $M_+$. 
In the same way, $S_t \cap M_-$ has no compressing disk in $H^1_t$. 

\begin{clm}\label{clm3}
There are no compressing disks for $S_t \cap M_-$ in $H^0_t$. 
\end{clm}
\begin{proof}
If such a compressing disk exists for $S_t \cap M_-$ in $H^0_t$, then it extends to a compressing disk in $S_{t'} \cap M_-$ in $H^0_{t'}$ for $t' > {\hat t}$ by Claim~\ref{clm1}. 
This gives a contradiction to the strong irreducibility of the splitting (resp. the assumption that $S_{t'} \cap M_-$ does not have compressing disks on both sides) 
in the case that the first compressing disk at the level $t' > \hat t$ appears in the $M_+$ side (resp. the $M_-$ side). 
\end{proof}

Therefore we conclude that, for $t<\hat t$ and $t$ arbitrarily close to $\hat t$, there are no compressing disks for either $S_t \cap M_+$ or $S_t \cap M_-$.
\end{proof}

We note that if the first option in the theorem occurs, then the Hempel distance of the Heegaard splitting is at most 2. 
This immediately implies the following. 

\begin{cor}\label{cor1}
Suppose that $J$ is separating and incompressible and $S$ is a Heegaard splitting which has Hempel distance at least $3$. Then the second or third possibilities must occur. 
\end{cor}

The next corollary gives a proof of the second assertion of Theorem~\ref{MainThm}. 

\begin{cor}\label{cor2}
Suppose that $J$ is separating and incompressible and $S$ is a Heegaard splitting which has Hempel distance at least $4$. Then only the third possibility can occur. 
\end{cor}

\begin{proof}
The first possibility in the theorem contradicts Hempel distance at least $3$ by Corollary~\ref{cor1}

Recall that the second case occurs when a single band sum of $S_t$ across $J$ at the critical level $\hat t$ produces a compressing disk $D_0$ in $H^0_t$ for $S_t$ and $t<\hat t$, whereas there is a compressing disk $D_1$ for $S_{t^\prime}$ in $H^1_{t'}$ for $t^\prime >\hat t$. 
There are either one or two curves of $S_t \cap J$ involved with the band sum. 
After the band sum, we get a new family of curves which can be pushed off the old family. 
But then we see that there is a compressing disk $D_0$ for $S_t$ in $H^0_t$ disjoint from $S_t \cap J$ for $t<\hat t$ and similarly a compressing disk $D_1$ for $S_{t^\prime}$ in $H^1_{t'}$ for $t^\prime >\hat t$. 
We conclude that $\partial D_0$ is disjoint from $S_t \cap J$ which can be made disjoint from $S_{t^\prime} \cap J$ which is disjoint from $\partial D_1$. 
This contradicts the Hempel distance of $S$ being at least $4$. 
\end{proof}

Also the following is deduced from our theorems. 

\begin{thm}
Suppose that a closed orientable 3-manifold $M$ has a strongly irreducible Heegaard splitting $S$ of Hempel distance at least 3 and a separating incompressible surface $J$. 
Then $M$ has a very short hierarchy consisting of $J$ and a collection of incompressible and boundary incompressible surfaces $\Sigma_+$ and $\Sigma_-$ properly embedded in $M_+$ and $M_-$ respectively, the two components of $M$ cut open along $J$. 
So when $M$ is cut open along $J, \Sigma_+, \Sigma_-$, the result is a collection of handlebodies. 
The sum of the Euler characteristics of the surfaces in $\Sigma_+, \Sigma_-$ is greater than or equal to the Euler characteristic of $S$ minus 1. 
In addition, if the Hempel distance of $S$ is at least 4, then the sum of the Euler characteristics of such surfaces in $\Sigma_+, \Sigma_-$ is greater than or equal to the Euler characteristic of $S$. 
\end{thm}

\begin{proof}
If the Hempel distance of $S$ is at least 3, by Corollary~\ref{cor1}, only the second or third possibilities in Theorem~\ref{MainThm} can occur. 
In addition, if it is at least 4, by Corollary~\ref{cor2}, only the third possibility can occur. 

When the second possibility occurs, there is a critical level $\hat t$ so that $S_{t} \cap M_+$ is incompressible for $t<{\hat t}$ and $t$ close to ${\hat t}$, and $S_{t'} \cap M_-$  is incompressible for $t' > {\hat t}$ and $t'$ close to ${\hat t}$, or the same with $M_+$, $M_-$ interchanged. 
Then $S_{t} \cap M_+$ and $S_{t'} \cap M_-$ split $M_+$, $M_-$ respectively into handlebodies, since $S_t$, $S_{t'}$ bound handlebodies $H^0_t$, $H^1_{t'}$ in $M$, and families of incompressible surfaces $J \cap H^0_t$, $J \cap H^1_{t'}$ split handlebodies into handlebodies. 
To form a very short hierarchy, it suffices to perform boundary compressions of the components of $S_{t} \cap M_+$ and $S_{t'} \cap M_-$ in $M_+$ and $M_-$ respectively. 
Notice these boundary compressions may remove some of the handles of the handlebodies. 
The result is another family of handlebodies, after we cut $M_+$ and $M_-$ respectively open along $\Sigma_+$ and $\Sigma_-$ which are the families of incompressible and boundary incompressible surfaces formed by the boundary compressions. 
Any component of $S_t \cap M_+$ and $S_{t'} \cap M_-$ which is boundary parallel in $M_+$ or $M_-$ respectively is discarded in this process. 
Note that 
there could be an extra band on $S$ between the cut open surfaces $S_t \cap M_+$ and $S_{t'} \cap M_-$. 
It then follows that $\chi(S) \le \chi(\Sigma_+) + \chi (\Sigma_-) +1$. 

When the third possibility occurs, there is a non-critical value of $t$ so that $S_t \cap M_+$ and $S_t \cap M_-$ are both incompressible. 
The same argument as above can be applied also is this case. 
Moreover, in this case, since the union of $S_t \cap M_+$ and $S_t \cap M_-$ is homeomorphic to $S$, $\chi(S) \le \chi(\Sigma_+) + \chi (\Sigma_-)$ holds. 
\end{proof}

\bigskip

\noindent{\bf Acknowledgements.}

The authors would like to thank Toshio Saito for useful conversations.

\end{document}